\documentclass[draft]{amsart}
\usepackage{amsmath,amsfonts,amssymb,cite,enumerate}
\usepackage{color}

\newcommand{\N}{\mathbb{N}}                   
\newcommand{\Z}{\mathbb{Z}}                   
\newcommand{\Q}{\mathbb{Q}}                   
\newcommand{\R}{\mathbb{R}}                   
\newcommand{\RR}{\mathcal{R}}                 
\newcommand{\trdeg}{\mathrm{trdeg}}           

\newtheorem{theorem}{Theorem}[section]
\newtheorem{proposition}[theorem]{Proposition}
\newtheorem{lemma}[theorem]{Lemma}
\newtheorem{corollary}[theorem]{Corollary}

\theoremstyle{definition}
\newtheorem{definition}[theorem]{Definition}

\newtheorem{remark}[theorem]{Remark}

\numberwithin{equation}{section}

\begin{document}
\title[On polynomially bounded reducts of $\R_{\mathrm{an},\exp}$]{Two remarks on\\ polynomially bounded reducts\\ of the restricted analytic field\\ with exponentiation}

\author{Serge Randriambololona}

\address{S. Randriambololona\\ Laboratoire de Math\'ematiques, Universit\'e ́de Savoie, UMR 5127 CNRS\\ 73376 Le Bourget-du-Lac C\'edex, FRANCE}

\email{serge.randriambololona@univ-savoie.fr}


\thanks{This work was pursued while the author was financially supported by the Institut de Recherche Math\'ematiques de Rennes (UMR 6625), the Laboratoire de Math\'ematiques de l'Universit\'e de Savoie (UMR 5127), and the ANR ``Singularit\'es de Trajectoires de Champs de Vecteurs Analytiques et Alg\'ebriques'' (ANR 11-BS01-0009).\\ 
}

\begin{abstract}
This article presents two constructions motivated by a conjecture of L. van den Dries and C. Miller concerning the restricted analytic field with exponentiation. The first construction provides an example of two o-minimal expansions of a real closed field that possess the same field of germs at infinity of one-variable functions and yet define different global one-variable functions. The second construction gives an example of a family of infinitely many distinct maximal polynomially bounded reducts (all this in the sense of definability) of the restricted analytic field with exponentiation.
\end{abstract}

\maketitle

\section{Introduction}

\label{sec:intro}
Properties of $\R_{\mathrm{an},\exp}$, the real exponential field with restricted analytic functions, have been widely studied since the mid-90's (starting with L. van den Dries and C. Miller's \cite{vdDMi94} and L. van den Dries, A. Macintyre, and D. Marker's \cite{vdDMacMar94}). 

Of particular interest are the properties of $\R_{\mathrm{an},\mathrm{Pow}}$, the real field with power functions and restricted analytic functions, which is a reduct, in the sense of definability\footnote
{Most definitions are not recalled in this Section, in order to make the introduction lighter. We assume that the reader is familiar with the terminology of model theory (see for instance \cite[Chapters 1-5]{bookPoiz00} of B. Poizat) and with o-minimality (see for instance \cite{bookvdD98} of van den Dries); less standard notions (such as what we mean by ``{\it in the sense of definability}'') are made precise in Section \ref{sec:gemvsfun} and \ref{sec:manymax}.},
of $\R_{\mathrm{an},\exp}$. In \cite{Mi94b}, C. Miller studies the theory of $\R_{\mathrm{an},\mathrm{Pow}}$ and proves, among other things, that $\R_{\mathrm{an},\mathrm{Pow}}$ is polynomially bounded (and in particular is a \emph{proper} reduct, in the sense of definability, of  $\R_{\mathrm{an},\exp}$).

In \cite{vdDMi96}, L. van den Dries and C. Miller conjecture that the structure $\R_{\mathrm{an},\mathrm{Pow}}$ is maximal among the polynomially bounded reducts of $\R_{\mathrm{an},\exp}$ (all this in the sense of definability).

\smallskip

An important partial answer is given independantly in \cite[Proposition 5.1]{So02} by R. Soufflet and in \cite[Corollary 2]{KuhKuh} by F.-V. Kuhlmann and S. Kuhlmann: they prove that if $\R_{\mathcal F}$ is a proper reduct, in the sense of definability, of $\R_{\mathrm{an},\exp}$ that is also an expansion, in the sense of definability, of  $\R_{\mathrm{an},\mathrm{Pow}}$, then $\R_{\mathcal F}$ and $\R_{\mathrm{an},\mathrm{Pow}}$ define the same subsets of $\R^2$. If $\R_{\mathrm{an},\mathrm{Pow}}$ is not maximal among the strict reducts of $\R_{\mathrm{an},\exp}$ (in the sense of definability), then a set witnessing this non-maximality needs to be of arity $\geq 3$.
\smallskip

As was noted by the author in \cite{Ra1}, two o-minimal expansions of the real field may define the same subsets of $\R^2$, while the first is a strict reduct, in the sense of definability, of the second. However this phenomenum can not appear in a saturated setting: \cite[Lemma 4.7.]{vdD98} of van den Dries insures that if $R_{\mathcal L_0}$ is a reduct of $R_{\mathcal L_1}$, each of the structure $R_{\mathcal L_0}$ and $R_{\mathcal L_1}$ being an $\omega$-saturated expansion of an o-minimal ordered group, and if the structures $R_{\mathcal L_0}$ and $R_{\mathcal L_1}$ define (with parameters) the same sets of arity $2$, then they define the same sets in any arity. 

Hence, if the maximality result for the collection of one variable functions established in \cite{KuhKuh} and \cite{So02} could be transfered from the real setting to an $\omega$-saturated setting, the correctness of the conjecture of van den Dries and Miller would follow.
\smallskip

In their original form, the results of \cite{KuhKuh} actually hold not only for expansions of the reals but also for $\omega$-saturated structures. Let $R_{\mathrm{an},\exp}$ be any model of the theory of $\R_{\mathrm{an},\exp}$ (in the language $\mathcal L_{\mathrm{an},\exp}$ with relational symbols for each subset of $\R^n$ definable in the real exponential field with restricted analytic functions) and let $R_{\mathrm{an},\mathrm{Pow}}$ be its reduct to the language $\mathcal L_{\mathrm{an},\mathrm{Pow}}$ (the sub-language of $\mathcal L_{\mathrm{an},\exp}$ with relational symbols for each subset of $\R^n$ definable in $\R_{\mathrm{an},\mathrm{Pow}}$).
Given a reduct $R_{\mathcal F}$ of $R_{\mathrm{an},\exp}$, let $H(R_{\mathcal F})$ denote the set of germs of at $+\infty$ of one-variable functions definable in $R_{\mathcal F}$ with parameters (the set $H(R_{\mathcal F})$ being viewed as a subset of (the Hardy field) $H(R_{\mathrm{an},\exp})$). \cite[Corollary 2]{KuhKuh} states that if $R_{\mathcal F}$ is a proper reduct of $R_{\mathrm{an},\exp}$ and if, at the same time, $R_{\mathcal F}$ is an expansion of $R_{\mathrm{an},\mathrm{Pow}}$, then $H(R_{\mathcal F})=H(R_{\mathrm{an},\mathrm{Pow}})$.
\smallskip

For two o-minimal structures over the reals, the local compactness of the real line insures the equivalence between the fact of having the same germs of one-variable functions at infinity and the fact of defining the same subsets of $\R^2$. It is therefore natural to wonder if this property still holds for structures over a general real closed field. 

The object of {\bf Section \ref{sec:gemvsfun}} is to show that this is not the case in general. We exhibit two o-minimal expansions of a common non-Archimedean real closed field that define the same germs at infinity of one variable functions, while not defining the same global one-variable functions. 
\smallskip

The results in {\bf Section \ref{sec:manymax}} are independent of those of Section \ref{sec:gemvsfun} but are also motivated by the conjecture of van den Dries and Miller; the techniques used in both Sections are, furthermore, similar. We show that there are many different maximal polynomially bounded reducts of $\R_{\mathrm{an},\exp}$: the maximality of $\R_{\mathrm{an},\mathrm{Pow}}$ remains open but there is no hope for $\R_{\mathrm{an},\mathrm{Pow}}$ to be the \emph{greatest element} among the polynomially bounded reducts of $\R_{\mathrm{an},\exp}$ (all this taken in the sense of definability).


\section{Germs versus functions}
\label{sec:gemvsfun}

In this Section, we present two o-minimal expansions of a non-Archimedean real closed field $\RR$, that define (with parameters) the same germs of one-variable functions at infinity but which do not define the same global functions in one variable.


\begin{definition}
\label{defi:structures}
A function $f:\R^n \to \R$ is said to be a restricted analytic function if there is a function $F$ analytic in a neighbourhood of $[0,1]^n$ such that $f(x)=F(x)$ for $x\in [0,1]^n$ and $f(x)=0$ for $x\notin [0,1]^n$. 
\end{definition}

Let $\RR$ be the field of Puiseux series ({\it i.e.} the direct limit of all the fields of formal Laurent series in $T^ {1/d}$ as $d$ ranges over $\N$). Considering $T$ as an infinitesimal, $\RR$ can be regarded as an ordered field extension of $\R$, the order on $\RR$ being defined by 
\[
(\zeta=\sum_{k=k_0}^{\infty}a_k T^{k/d} \wedge a_{k_0}> 0) \Leftrightarrow \zeta>0.
\]

Following \cite[Section 2.]{vdDMacMar94}, one can extend any restricted analytic function $f:\R \to \R$ to a function $\widetilde f: \RR \to \RR$ : let $U$ be an open neighbourhood of $[0,1]$ and $F:U\to \R$ be an analytic function such that $f|_{[0,1]}=F|_{[0,1]}$ and consider $\zeta \in \RR$;
\begin{itemize}
\item if $\zeta<0$ or $\zeta>1$, let $\widetilde{f}(\zeta):=0$;
\item if $0\leq \zeta \leq 1$, let $\widetilde{f}(\zeta)$ be the formal composite of $F_{a_0}$ and $\rho(\zeta)$ where
\begin{itemize}
\item $a_0$ is the constant coefficient of the development of $\zeta$, 
\item $F_{a_0}$ is the (converging) Taylor development of $F$ at $a_0$ (which exist since $0\leq a_0 \leq 1$), 
\item $\rho (\zeta)=\zeta-a_0$.
\end{itemize}
\end{itemize}
(It is possible to extend in a similar manner a restricted analytic function of several variables; we however only need the one-variable case in what follows.)

\begin{definition}
\label{defi:Rf}
Let $f:\R \to \R$ be a restricted analytic function and  $\widetilde f:\RR \to \RR$ be its extension to the fields of Puiseux series described above.

We will denote by $\R_f$ the structure \[
\R_f:=(\R;<,+,\cdot \, ,f)
\]
 and by $\RR_f$ the structure 
 \[
 \RR_f:=(\RR;<,+,\cdot \, ,\widetilde{f}).
 \]
\end{definition}

\begin{remark}
\label{rem:elementary}
The structure $\R_f$ is o-minimal. As noted in \cite[Corollary 2.11]{vdDMacMar94}, it is also an elementary substructure of $\RR_f$; that is, if $\phi(x_1,\ldots,x_n)$ is a first order logic formula in the language $\mathcal L_f=\{ <, +,\cdot \, ,f \}$ and $(a_1,\ldots,a_n) \in \R^n$, the property $\phi(a_1,\ldots, a_n)$ holds true when interpreted in $\R_f$ if and only the property $\phi(a_1,\ldots, a_n)$ holds true when interpreted in $\RR_f$.
\end{remark}

\begin{definition}
\label{defi:cut}
Let $\kappa$ be the generalized power series 
\[
\frac{1}{2}+\sum_{k=1}^{\infty} T^{k+\frac{1}{k}}.
\] 
For $\zeta \in \RR$, we will write 

\begin{itemize}
\item $\zeta < \kappa$ if $\zeta< \frac{1}{2}+\sum_{k=1}^K T^{k+\frac{1}{k}}$ for some $K\in \N$,
\item $\zeta > \kappa$ if $\zeta> \frac{1}{2}+\sum_{k=1}^K T^{k+\frac{1}{k}}$ for all $K\in \N$.
\end{itemize}

 This defines a Dedekind cut on $\RR$. 
 
We chose $\kappa$ so that the $1$-type over $\RR$ associated to this cut is not definable. In particular, if $\zeta\in \RR$ and  $\zeta<\kappa$ (respectively $\zeta>\kappa$), there is $\xi\in \RR$ such that $\zeta<\xi<\kappa$ ({\it resp.}  $\zeta>\xi>\kappa$).
\end{definition}

\begin{definition}
\label{defi:Rfk}
Let $\widetilde{f}:\RR \to \RR$ be as in Definition \ref{defi:Rf}. 
Under the notations of Definition \ref{defi:cut}, $\RR_{f|_{\kappa}}$ will denote the structure 
\[
\RR_{f|_{\kappa}}:=(\RR;<,+,\cdot,(\widetilde f|_{[0,a]})_{a<\kappa}), (\widetilde f|_{[b,1]})_{b>\kappa})).
\]

(Out of convenience, we identify any partial function $g:\RR \to \RR$ to a total function by setting $g(x)=0$ for $x$ outside of the original domain of $g$.) 

\end{definition}

We can now state the first result of this section :

\begin{proposition}
\label{prop:samegerms}
For any function $g:\RR\to \RR$ definable in $\RR_f$ (with parameters), there is a positive $\varepsilon\in \RR$ such that $g|_{(0,\varepsilon)}$ is definable in $\RR_{f|_{\kappa}}$.
\end{proposition}

Once this proposition established, we will need to choose $f$ so that $\RR_f$ defines strictly more sets than $\RR_{f|_{\kappa}}$ does.

\medskip
Recall the following definition from O. Le Gal's \cite{lG10} :

\begin{definition}
\label{defi:strongtranscendence}
A function $f:\R \to \R$ is said to be a strongly transcendental restricted $\mathcal C^{\infty}$ function if  $f(x)=0$ for all $x\notin [0,1]$ and $f(x)=F(x)$ for all $x\in [0,1]$ where 
\begin{itemize}
	\item $F:U\to \R$ is a $\mathcal C^{\infty}$ function in some neighbourhood $U$ of $[0,1]$ and
	\item given any tuple $x=(x_1, \ldots , x_n)$ of pairwise distinct elements of $U$, there exists a constant $C\in \N$ such that, for all $m\in \N$, the transcendence degree over $\Q$ of the $n(m+2)$-tuple 
	\[
	(x_1,\ldots,x_n,F(x_1),\ldots,F(x_n), \ldots , F^ {(m)}(x_1), \ldots ,F^ {(m)}(x_n))
	\]
	is higher than $n(m+2)-C$.
\end{itemize}
\end{definition}

Following \cite{lG10}, if $x$ denotes the $n$-tuple $(x_1,\ldots,x_n)$, the notation $j^ m_nF(x)$ denotes the $n(m+1)$-tuple $(F(x_1),\ldots,F(x_n), \ldots , F^ {(m)}(x_1), \ldots ,F^ {(m)}(x_n))$; the notation $\trdeg(x_1,\ldots,x_n)$ denotes the transcendence degree of $x$ over $\Q$.

\begin{proposition}
\label{prop:nof}
Under the notations of Definition \ref{defi:Rf} and \ref{defi:Rfk}, if $f$ is a restricted analytic function which is also a restricted strongly transcendental function then the function $\widetilde f$ is not definable in $\RR_{f|_{\kappa}}$.
\end{proposition}

\begin{remark}
Note that the assumption on $f$ made in the hypothesis of Proposition \ref{prop:nof} is non-vacuous:  \cite[Proposition 2.2.]{lG10} insures that there exist (many) restricted analytic, strongly transcendental functions.
\end{remark}

Propositions \ref{prop:samegerms} and \ref{prop:nof} imply as announced:

\begin{theorem}
There exists a pair of o-minimal expansions of a common non-Archimedean field that do possess the same set of germs at infinity of one-variable definable (with parameters) functions but do not possess the same set of global definable (with parameters) one-variable functions.
\end{theorem}
\medskip

\begin{proof}[Proof of Proposition \ref{prop:samegerms}]

Let $g$ be definable in $\RR_f$ with some parameters $\beta\in \RR^p$.

Up to compositions with $\emptyset$-definable Nash bijection between $(0,1)$ and $\R$, we can find a $\emptyset$-definable function $G$ from $[0,1]^{p+1}$ to $\R$ such that $g(x)=\widetilde G(\beta,x)$, where $\widetilde G$ is the interpretation of $G$ in $\RR_f$ (see Remark \ref{rem:elementary}).

By the syntactic version of Gabrielov's Theorem of the complement (\cite[Corollary]{Gab}), there is some $q\in \N$, some set $X\subset [0,1]^p\times [0,1]^2\times [0,1]^q $ such that 
the graph of $G$ is $\pi (X)$ where $\pi$ denotes the projection on the first $p+2$ coordinate axes and such that $X$ is described by a finite boolean combination of formul\ae \ of the form 
\[
P(y_1,\ldots,y_{p+2+q},f(y_1),\ldots,f(y_{p+2+q}), \ldots , f^ {(m)}(y_1), \ldots ,f^ {(m)}(y_{p+2+q}))=0
\] and
\[
Q(y_1,\ldots,y_{p+2+q},f(y_1),\ldots,f(y_{p+2+q}), \ldots , f^ {(m)}(y_1), \ldots ,f^ {(m)}(y_{p+2+q}))>0
\]
for $P$ and $Q$ some polynomial with coefficients in $\Z$.

Let $\widetilde X$ be the interpretation of $X$ in $\RR_f$ and $\widetilde X_{\beta}$ be the fibre of $\widetilde X$ over $\beta$ for the projection $\RR^p\times \RR^2\times \RR^q \to \RR^2\times \RR^q$.

By Definable Choice (see \cite[Proposition 6.1.2.]{bookvdD98}), for $\varepsilon>0$ small enough, there is a definable function $\zeta:(0,\varepsilon)\to \widetilde X_{\beta}$ such that for all $0<x<\varepsilon$ one has $(x,g(x))=\pi' (\zeta (x))$ (where $\pi'$ denote the projection $\RR^p\times\RR^2\times \RR^q\to \RR^2$). Up to taking an even smaller $\varepsilon$, we can assume that each component $\zeta_i$ of $\zeta$ is continuous. If for each $1\leq i \leq 2+q$, we denote by $\xi_i=\lim_{s \to 0}\zeta(s)\in [0,1]$, we can further shrink $\varepsilon$ so that each set $L_i=\zeta_i((0,\varepsilon))$ is either a singleton or an open interval and its topological closure lies entirely in one side or the other of the cut $\kappa$ (the side depending on whether $\xi_i>\kappa$ or $\xi_i<\kappa$).

Let $\Gamma$ be the graph of $g_{|(0,\varepsilon)}$. We now have that \[
\Gamma=\pi'(\zeta((0,\varepsilon))\subset \pi'(\widetilde X_{\beta}\cap \prod_{i=1}^{2+q} L_i ) \subset \Gamma.
\]

Since, for each $i$, the topological closure of each $L_i$ lies in one side or the other of the cut $\kappa$, there is some $c_i$ such that, 
\begin{itemize}
\item either ($0\leq c_i <\kappa $ and $(\forall x\in \RR , (x\in L_i \rightarrow 0\leq x \leq c_i ))$)
\item or ($\kappa < c_i \leq 1$ and $(\forall x\in \RR , (x\in L_i \rightarrow c_i \leq x \leq 1))$).
\end{itemize}

The set $\widetilde X_{\beta}\cap \prod_{i=1}^{2+q} L_i$ being a boolean combination of sets of vanishing and sets of positivity of polynomial in the functions $(z_1,\ldots,z_{2+q})\mapsto z_i$ and $(z_1,\ldots,z_{2+q})\mapsto f_{|L_j}^ {(d)}(z_j)$ with coefficients in $\RR$, it is definable in $\RR_{f|_\kappa}$. It follows that $g_{|(0,\varepsilon)}$ is definable in $\RR_{f|_\kappa}$.
\end{proof}

Before proving Proposition \ref{prop:nof}, we need the following real version of it:

\begin{lemma}
\label{lem:nof}
Let $f:\R \to \R$ be a restricted analytic function. Assume furthermore that $f$ is a strongly transcendental restricted $\mathcal C^{\infty}$ function. Consider $(a,b)\in \R^2$ with $0<a<b<1$.

Then $f$ is not definable in the structure $(\R;\leq,+,\cdot \ , f|_{[0,a]}, f|_{[b,1]}))$. 
\end{lemma}

\begin{proof}[Proof of Lemma \ref{lem:nof}]
Suppose that $f$ is definable in $(\R;<,+,\cdot,f|_{[0,a]}, f|_{[b,1]})$ with some parameters. Let $g(x)=f(ax)$ and $h(x)=f(x+b(1-x))$.
By \cite[Lemma 3]{Gab}, we can find some $p\in \N$, a finite collection of subsets $X_{\nu}$ of $[0,1]^2\times [0,1]^p$ and a finite collection $V$ of points in $[0,1]^2\times [0,1]^p$ such that 
\begin{enumerate}
\item the graph of $g$ is the union of the projections on the first two coordinates of $V$ and of the $X_{\nu}$'s,
\item each $X_{\nu}$ is the intersection of the positivity set $P_\nu$ of a finite set $\Omega_{\nu}$ of functions, with the zero-set $Z_\nu$ of a finite set $\Theta_{\nu}$ of functions, where each function in $\Omega_{\nu}$ and $\Theta_{\nu}$ is given as a polynomial with real coefficients in the functions $(z_1,\ldots,z_{2+q})\mapsto z_i$, $(z_1,\ldots,z_{2+q})\mapsto g^ {(d)}(z_j)$ and $(z_1,\ldots,z_{2+q})\mapsto h^ {(e)}(z_k)$, 
\item for each $\nu$, the set $X_{\nu}$ is an analytic manifold of dimension $1$ given near each of its points by the transverse intersection of analytic hypersurfaces defined by each function in $\Theta_{\nu}$ and 
\item the projection on the first $2$ coordinates has full rank $1$ when restricted to each $X_{\nu}$.
\end{enumerate}

The projection of $V$ being finite, we can find  some $c\in \R$ and $\varepsilon>0$ such that $(c-\varepsilon,c+\varepsilon)\subset (a,b)$ and the set $\{(x,y)\in \R^ 2 ; c-\varepsilon<x<c+\varepsilon , y=f(x) \}$ is the image by the projection $\pi:[0,1]^2\times [0,1]^p \to [0,1]^2$ of an analytic manifold $\Gamma$ given on some open set $U\subset [0,1]^2\times [0,1]^p$ as the conjunction of $p+1$ transverse smooth hypersurfaces of the form
\begin{equation*}
\{z\in U \ ; \ P(z,j^{\ m}_{2+q}g(z),j^{\ m}_{2+q}h(z))\}
\end{equation*} 
for some polynomial $P$ and so that $\pi_{|\Gamma }$ is a one-to-one submersion between $\Gamma$ and the graph of the restriction of $f$ to $(c-\varepsilon,c+\varepsilon)$.

Let $\gamma$ be the preimage of $(c,f(c))\in \R^2$ by $\pi_{|\Gamma }$ and let $\beta$ be a tuple made of the coefficients involved in the different polynomials $P$ used to describe $\Gamma$ in $U$.

By the chain rule and an easy induction, we can find, for all $D\in \N$, a rational function $\Phi_D$ with rational coefficients such that 
\[
j_1^D f(c)=\Phi^ D (\beta ,\gamma , j_{n+p}^ {D+m} g(\gamma),j_{n+p}^{D+m}h(\gamma)) 
\]

Let $\eta$ be a $s$-tuple whose coordinates are all the different images of the coefficients of $\gamma$ by the map $x \mapsto ax$ and 
$x \mapsto x+b(1-x)$. 
Then  for all $D\in \N$ there is  a rational function $\Psi_D$ with rational coefficients such that  
\begin{equation}
\label{eqn:dependant}
j_1^D f(c)=\Psi^ D (a,b,\beta ,\gamma, \eta , j_{n+p}^ {D+m} f(\eta)) 
\end{equation}

Since $c\in (a,b)$, $c$ is not a coordinate of $\eta$. The function $f$ being strongly transcendental, there is $C\in \N$ such that for all $D\in \N$, 
\begin{eqnarray*}
(s+1)(D+1)-C &\leq& \trdeg (c,j^D_1f(c),\eta , j_s^ D f(\eta))\\
&\leq &	\trdeg (c,j^D_1 f(c),\eta ,j_s^{D+m} f(\eta), a, b, \beta ,\gamma).
\end{eqnarray*}
But by the inequality \ref{eqn:dependant},
\[
\trdeg (c,j^D_1f(c),\eta ,j_s^{D+m} f(\eta ),a,b,\beta ,\gamma)=\trdeg (c,\eta ,j_s^{D+m} f(\eta ),a,b,\beta ,\gamma)
\]
so that
\[
(s+1)(D+1)-C \leq s(D+m+1)+ \trdeg (c,\eta,a,b,\beta ,\gamma ).
\]
 
However, the latter inequality can not hold for large integers $D$ : this is a contradiction.

\end{proof}

\begin{proof}[Proof of Proposition \ref{prop:nof}]
Generalizing Lemma \ref{lem:nof} to $\RR$ is an easy syntactic manipulation.

Suppose for a contradiction that $f$ is definable in $\RR_{f|_{\kappa}}$. By finiteness of first order logic formul\ae, $f$ is definable in the structure $(\RR;\leq,+,\cdot \ , \widetilde f|_{[0,a]}, \widetilde f|_{[b,1]}))$ for some $a$ and $b$ in $\RR$ with $0<a<\kappa <b<1$. 

Let $\mathcal L_{f,g,h}$ be the expansion of the real ordered field language obtained by adding three extra functional symbol of arity $1$ (denoted, without ambiguity, $f$, $g$ and $h$) and $\mathcal L_{f}$ (respectively $\mathcal L_{g,h}$) be its reduct obtained by removing the symbols $g$ and $h$ ({\it resp.} the symbol $f$) and let $\RR_{f,g,h}$ be the $\mathcal L_{f,g,h}$-expansion of the real closed field $\RR$ in which $f$ (respectively $g$ and $h$) is interpreted by $\widetilde f$ ({\it resp.} $\widetilde f|_{[0,a]}$ and $\widetilde f|_{[b,1]}$). 

We then have
\[
\RR_{f,g,h} \models \exists \beta  \big( (y=f(x)) \leftrightarrow \phi_{g,h} (x,y,\beta) \big)
\]
where $\phi_{g,h}$ is an $\mathcal L_{g,h}$-formula.

We can add new existential quantifiers so that each atomic formula appearing in the formula $\phi_{g,h} (x,y,\beta)$ is either in the pure language of rings or is of one of the forms $v=g(u)$ or $v=h(u)$ for some variables $u$ and $v$.

Let $a$ and $b$ be two distinguished variables and let $\phi _f(x,y,a,b,\beta)$ be the $\mathcal L_f$ formula obtained by replacing  in $\phi_{g,h}(x,y,\beta)$
\begin{itemize}
\item each atomic formula of the form ``$ v=g(u)$'' by a formula of the form ``$(0\leq u \leq a \wedge v=f(u)) \vee v=0$'' and 
\item each atomic formula of the form ``$ v=h(u)$'' by a formula of the form ``$(b\leq u \leq 1 \wedge v=f(u)) \vee v=0$''.
\end{itemize}

Then 
\[
\RR_f \models \exists a \exists b \exists \beta  (0<a<b<1) \wedge \big( (y=f(x)) \leftrightarrow \phi_f (x,y,a,b,\beta) \big)
\]
and since $\R_f$ is an elementary substructure of $\RR_f$ (as noted in Remark \ref{rem:elementary}), 
\[
\R_f\models \exists a \exists b \exists \beta  (0<a<b<1) \wedge \big( (y=f(x)) \leftrightarrow \phi_f (x,y,a,b,\beta) \big)
\]
which contradicts Lemma \ref{lem:nof}.
\end{proof}

\begin{remark}
Note that the question of whether Hardy fields of germs at infinity of one variable functions determine the structure was asked with the hope of combining \cite[Lemma 4.7.]{vdD98} and \cite[Corollary 2]{KuhKuh}. In the example presented in this section, even though  we can easily replace $\RR_f$ by an $\omega$-saturated, $\kappa$ and $\RR_{f|_{\kappa}}$ are chosen precisely so that the structure $\RR_{f|_{\kappa}}$ is \emph{not} $\omega$-saturated.

Consider $\mathfrak R_{f,f|_{\kappa}}$ an $\omega$-saturated elementary expansion of the structure 
\[
(\RR;<,+,\cdot,\widetilde{f}, (\widetilde f|_{[0,a]})_{a<\kappa}), (\widetilde f|_{[b,1]})_{b>\kappa})).
\] 
No analogue of Proposition \ref{prop:samegerms} holds for the reducts $\mathfrak{R}_f$ and $\mathfrak{R}_{f|_{\kappa}}$ of $\mathfrak R_{f,f|_{\kappa}}$: there is a realisation $\chi \in \mathfrak R$ of the type $\kappa$ and the germ at $\chi$ of the realisation of $f$ is not the germ of a function definable in the structure $\mathfrak R_{f|_{\kappa}}$, precisely by the analogue of Proposition \ref{prop:nof}.
\end{remark}

\section{No greatest element}
\label{sec:manymax}

In this section, we show that there are infinitely many polynomially bounded structures $(\R _{\mathcal F _n})_{n\in \N}$ which are pairewise distinct maximal reducts of the restricted analytic field with exponentiation (all this in the sense of definability).

But first, let's make precise what we mean by ``{\it in the sense of definability}''.

\begin{definition}
\label{defi:defi}
Given two structures $\mathcal M_0=(M;\cdots )$ and  $\mathcal M_1=(M;\cdots )$ on the same universe $M$, we say that $\mathcal M_0$ is a (strict) reduct, in the sense of definability, of $\mathcal M_1$ (or that $\mathcal M_1$ is a (strict) expansion, in the sense of definability, of $\mathcal M_0$) if $\mathcal M_0$ defines, with parameters, (strictly) less sets than does $\mathcal M_1$. 

Note that the fact that $\mathcal M_0$ is a reduct, in the sense of definability, of $\mathcal M_1$ does not imply the fact that $\mathcal M_0$ is a reduct, in the classical sense, of $\mathcal M_1$; note also that $\mathcal M_0$ can be a strict reduct of $\mathcal M_1$ in the classical sense whithout being a strict reduct in the sense of definability.
\end{definition}

\begin{definition}
\label{defi:PB}
Recall that an expansion of the real field is said to be polynomially bounded if whenver $f$ is a one-variable definable function, $f(x)$ grows at most as fast as a polynomial function as $x$ goes to $+\infty$ ({\it i.e.} there is some $d\in \N$ such that $\exists M, (x>M\rightarrow |f(x)|\leq x^d)$). 
\end{definition}

Polynomial boundedness is an important dividing line among o-minimal expansions of the reals. The Growth Dichotomy Theorem of \cite{Mi94a} states that polynomial boundedness is a necessary and \emph{sufficient} condition for an o-minimal expansion of the real field not to define the exponential function. (Note that \cite{KarMac99} insures that given an o-minimal expansion of the real field, one can always expand it further by adding the exponential, while keeping o-minimality.)

Our main subject of study are following:

\begin{definition}
We denote by $\R_{\mathrm{an}}$ the expansion of the real field by all restricted analytic functions (see Definition \ref{defi:structures}), by $\R_{\mathrm{an},\exp}$ the expansion of $\R_{\mathrm{an}}$ by the exponential function and by $\R_{\mathrm{an},\mathrm{Pow}}$ the expansion of $\R_{\mathrm{an}}$ by all the power functions (functions $f_r:\R \to \R$ defined by $f_r(x)=x^r$ if $x> 0$, $f_r(x)=0$ if $x\leq 0$).
\end{definition}

The structure $\R_{\mathrm{an}}$ is o-minimal and polynomially bounded (following important results from A. Khovaskii, S. {\L}ojasiewicz and A. Gabrielov) and its expansion $\R_{\mathrm{an},\exp}$ is still o-minimal (as first proved in \cite{vdDMi94}). The structure $\R_{\mathrm{an},\mathrm{Pow}}$ is a strict reduct, in the sense of definability, of $\R_{\mathrm{an},\exp}$ but a strict expansion, in the sense of definability, of $\R_{\mathrm{an}}$  (by \cite{Mi94}).

As recalled in the introduction, van den Dries and Miller conjecture in \cite{vdDMi96} that $\R_{\mathrm{an},\mathrm{Pow}}$ is maximal among the polynomially bounded reducts of $\R_{\mathrm{an},\exp}$ (all this in the sense of definability).
\smallskip

Relying on results from \cite{lG10}, we prove the existence of an infinite collection of $(\R _{\mathcal F _n})_{n\in \N}$ of maximal polynomially bounded expansions of the real field which are strict reducts of $\R _{\mathrm{an},\exp}$  (all this in the sense of definability).

Recall first :

\begin{theorem}[{\cite[Theorem 1.2.]{lG10}}]
For each $f:\R\to \R$ strongly transcendental restricted $\mathcal C^{\infty}$ function, the structure $\R_f:=(\R;\leq,+,\cdot ,f)$ is o-minimal and polynomially bounded.
\end{theorem}

See Definition \ref{defi:strongtranscendence}; note that in this Section, contrary to Section \ref{sec:gemvsfun}, the function $f$  is not furthermore required to be restricted analytic.
\smallskip

Next result, also from \cite{lG10}, states that the set of strongly transcendent $\mathcal C^{\infty}$ functions is hard to avoid. Let $\mathcal A$ be the set of restrictions to $[0,1]$ of functions which are analytic in a neighbourhood of $[0,1]$, with radius of convergence $\geq 1$ at each point of  $[0,1]$. The norm $\| f \| = \sup_{x\in [0,1]} \frac{|F^{(k)}(x)|}{k!}$ (where $F$ is any analytic continuation of $f$ to an open neighbourhood of $[0,1]$) turns $\mathcal A$ into a Banach space. Let $\mathcal S$ denote the set of strongly transcendental restricted $\mathcal C^\infty$ functions.

\begin{proposition}[{\cite[Proposition 2.2.]{lG10}}]
\label{prop:co-meagre}
Consider $h$ any function admitting a $\mathcal C^{\infty}$ continuation to an open neighbourhood of $[0,1]$. Then the set $\mathcal A \cap (h+\mathcal S)$ is co-meagre in $\mathcal A$.
\end{proposition}

As a corollary, we get:

\begin{corollary}
\label{cor:plenty}
Let $\varepsilon:[0,1]\to \R$ be the function defined by $\varepsilon(x)=e^{-1/x}$ if $0<x\leq 1$ and $\varepsilon(0)=0$. There is a function $f\in \mathcal A$ such that for all $n\in \N$, the function $f_n: x\mapsto f(x)+n\varepsilon(x)$ is a strongly transcendental restricted $\mathcal C^\infty$ function.
\end{corollary}

\begin{proof} The proof is straightforward. For each $n\in \N$, $\mathcal A \cap (-n\varepsilon +\mathcal S)$ is co-meagre in $\mathcal A$. But a countable intersection of co-meagre sets is also co-meagre. Therefore $ \mathcal A\ \cap\  \bigcap_{n\in \N} (-n\varepsilon +\mathcal S)$ is co-meagre in $\mathcal A$. In particular Baire Categoricity Theorem implies that  $ \mathcal A\ \cap\  \bigcap_{n\in \N} (-n\varepsilon +\mathcal S)$ is non-empty.

Let $f$ be in $ \mathcal A\ \cap\  \bigcap_{n\in \N} (-n\varepsilon +\mathcal S)$; then for each $n\in \N$, $f_n:x\mapsto f(x)+n\varepsilon(x)$ is strongly transcendental on $[0,1]$.
\end{proof}

\begin{theorem}
\label{theo:max}
There is a family $(\mathcal F_n)_{n\in \N}$ of collections $\mathcal F_n$ of functions definable in $\R _{\mathrm{an},\exp}$, such that, 
\begin{itemize}
\item for each $n$, the structure $\R_{\mathcal F_n}:=(\R;\leq,+,\cdot \ ,(g)_{g\in \mathcal F_n})$ is a maximal polynomially bounded reduct of $\R _{\mathrm{an},\exp}$  (in the sense of definability) and 
\item for each $n_1\neq n_2$, the structure $\R_{\mathcal F_{n_1}}$ and $\R_{\mathcal F_{n_2}}$ do not define the same sets.
\end{itemize}
\end{theorem}

\begin{proof}
For each fixed $n_0$, note that $f_{n_0}$ is definable in $\R _{\mathrm{an},\exp}$ and complete the singleton $\{ f_{n_0} \}$ to get a maximal set $\mathcal F_{n_0}$ of functions definable in $\R _{\mathrm{an},\exp}$ such that the structure $\R_{\mathcal F_{n_0}}:=(\R;\leq,+,\cdot \ ,(g)_{g\in \mathcal F_{n_0}})$ is polynomially bounded. 

The first conclusion of the Corollary is clearly satisfied. 

For the second conclusion of the Corollary, suppose $\R_{\mathcal F_{n_1}}$ defines $f_{n_2}$ with $n_1\neq n_2$, then $\R_{\mathcal F_{n_1}}$ defines $f_{n_2}-f_{n_1}=(n_2-n_1)\varepsilon$, contradicting the polynomial boundedness. 
\end{proof}

\begin{remark}
Note that given $n\in \N \setminus \{ 0 \}$ and $f_n$ as in Corollary \ref{cor:plenty}, the structure $\R_{\mathrm{an}, f_n}$ (obtained by expanding the restricted analytic field by the function $f_n$) defines the exponential: we have produced infinitely many polynomially bounded reducts of $\R_{\mathrm{an},\exp}$ but \emph{none of them} is an expansion of $\R_{\mathrm{an}}$ (all this in the sense of definability). If van den Dries and Miller's conjecture were to be proven true, it would follow that $\R_{\mathrm{an},\mathrm{Pow}}$ is the \emph{unique} maximal pollynomially bounded reduct of  $\R_{\mathrm{an},\exp}$ that expands $\R_{\mathrm{an}}$ (all this in the sense of definability): if $\R_\mathcal F$ is a maximal polynomially bounded reduct of  $\R_{\mathrm{an},\exp}$ that expands $\R_{\mathrm{an}}$ (in the sense of definability), then, by \cite[Result 3.2]{Mi94} and maximality, $\R_\mathcal F$ defines all power functions and is therefore an expansion, in the sense of definability, of $\R_{\mathrm{an},\mathrm{Pow}}$. 

Note also that the presentation of each $\R_{\mathcal F_n}$ is, in a double way, not constructive; firstly, the existence of a function $f$ as in Corollary \ref{cor:plenty} is a non-constructive result; secondly, once $f$ is chosen, the existence of each collection $\mathcal F_n$ is also given in a non-constructive way. This raises questions about elementary equivalence or isomorphism (in a certain sub-language $\mathcal L$ of $\mathcal L_{\mathrm{an},\exp}$ (conjecturally $\mathcal L_{\mathrm{an},\mathrm{Pow}}$)) of these structures, each seen as a reduct to the language $\mathcal L$ of an $\mathcal L_{\mathrm{an},\exp}$-structure over $\R$, bi-interpretable with the standard $\R_{\mathrm{an},\exp}$ (in the spirit of \cite[Theorem 2.1]{PiNe91}). 
\end{remark}

\bibliography{bibserge}
\bibliographystyle{plain}

\end{document}